\documentclass[a4paper,oneside,10pt]{article}%
\usepackage{amsmath}
\usepackage{amsfonts}
\usepackage{amssymb}
\usepackage{graphicx}
\usepackage[square,numbers,sort&compress]{natbib}
\usepackage[colorlinks,linkcolor=blue,anchorcolor=blue,citecolor=blue]%
{hyperref}%
\providecommand{\U}[1]{\protect\rule{.1in}{.1in}}

\newtheorem{theorem}{Theorem}[section]

\newtheorem{definition}[theorem]{Definition}

\newtheorem{lemma}[theorem]{Lemma}

\newtheorem{proposition}[theorem]{Proposition}
\newtheorem{remark}[theorem]{Remark}

\newenvironment{proof}[1][Proof]{\noindent \textbf{#1.} }{\  \rule{0.5em}{0.5em}}
\numberwithin{equation}{section}

\begin{document}

\title{Local time and Tanaka formula of $G$-martingales }
\author{Guomin Liu \thanks{Zhongtai Securities Institute for Financial Studies,
Shandong University, Jinan 250100, China, gmliusdu@163.com.}}
\maketitle

\textbf{Abstract}. The objective of this paper is to study the local time and
Tanaka formula of symmetric $G$-martingales. We introduce the local time of $G$-martingales and show that they belong to $G$-expectation space $L_{G}^{2}(\Omega
_{T})$. The bicontinuous modification of local time is obtained. We finally give the Tanaka formula for convex functions of $G$-martingales.


\textbf{AMS 2010 subject classifications:} 60H10, 60H30 \addcontentsline{toc}{section}{\hspace*{1.8em}Abstract}

\section{Introduction}
Motivated by probabilistic interpretations for fully nonlinear PDEs and
financial problems with model uncertainty, Peng \cite{peng2005,peng2008,Peng
1} systematically introduced the nonlinear $G$-expectation theory. Under the
$G$-expectation framework, Peng constructed $G$-Brownian motion, $G$-It\^{o}'s
stochastic calculus and $G$-stochastic differential equations ($G$-SDEs). The reader can refer to  \cite{G1,HJPS,Liny,Osuka} for further developments. 

One of the most important notion under $G$-framework is the $G$-martingles, which are defined as the processes satisfying martingale property through conditional $G$-expectation. The  representation theorem for $G$-martingales are obtained in \cite{STZ,Song,Song1}. The L\'{e}vy's characterization of $G$-martingales are investigated in \cite{XZ,XZ1}. The developments of $G$-martingales has a deep connection with the settlement of $G$-backward stochastic differential equations ($G$-BSDEs), see \cite{HJPS}.

This paper study the local time and Tanaka formula of symmetric $G$%
-martingales. It generalized the results in \cite{L-Q,H-J-L} where the
 $G$-Brownian motion case are considered. Compared with the classical case,  the integrand space for stochastic integral of  $G$-martingales is not big enough because of nonlinearity. So we first 
introduce a proper integrand space $\bar{M}_G^2(0,T)$ which is bigger than the previous $M_G^2(0,T)$ when the quadratic variation of $G$-martingales is
 degenerate.  Then, by proving some characterization results for $\bar{M}_G^2(0,T)$ and using the Krylov's estimate method as
in \cite{H-W-Z}, we construct the local time $L_{t}(a)$ for $G$-martingales and
show that $L_{t}(a)$ belong to the $G$-expectation space $L_{G}^{2}(\Omega
_{t})$.  Moreover, with the help of a localization argument, we prove that $L_{t}(a)$  has a modification which is continuous in $a$ and $t$. Finally,
we give the Tanaka formula for convex functions of $G$-martingales and state some basic properties of local time.

The paper is organized as follows. In Section 2, we recall some basic notions
and results of $G$-expectation and $G$-martingales. In Section 3, we state the
main results on local time and Tanaka formula of $G$-martingales.

\section{Preliminaries}

In this section, we review some basic notions and results of $G$-expectation
and $G$-martingales. More relevant details can be found in
\cite{peng2005,peng2008,Peng 1}.

\subsection{$G$-expectation space}

Let $\Omega$ be a given nonempty set and $\mathcal{H}$ be a linear space of
real-valued functions on $\Omega$ such that if $X_{1}$,$\dots$,$X_{d}%
\in\mathcal{H}$, then $\varphi(X_{1},X_{2},\dots,X_{d})\in\mathcal{H}$ for
each $\varphi\in C_{b.Lip}(\mathbb{R}^{d})$, where $C_{b.Lip}(\mathbb{R}^{d})$
is the space of bounded, Lipschitz functions on $\mathbb{R}^{d}$.
$\mathcal{H}$ is considered as the space of random variables.

\begin{definition}
A sublinear expectation $\hat{\mathbb{E}}$ on $\mathcal{H}$ is a functional
$\mathbb{\hat{E}}:\mathcal{H}\rightarrow\mathbb{R}$ satisfying the following
properties: for each $X,Y\in\mathcal{H}$,

\begin{description}
\item[\textrm{(i)}] {Monotonicity:}\quad$\mathbb{\hat{E}}[X]\geq
\mathbb{\hat{E}}[Y]\ \ \text{if}\ X\geq Y$;

\item[\textrm{(ii)}] {Constant preserving:}\quad$\mathbb{\hat{E}%
}[c]=c\ \ \ \text{for}\ c\in\mathbb{R}$;

\item[\textrm{(iii)}] {Sub-additivity:}\quad$\mathbb{\hat{E}}[X+Y]\leq
\mathbb{\hat{E}}[X]+\mathbb{\hat{E}}[Y]$;

\item[\textrm{(iv)}] {Positive homogeneity:}\quad$\mathbb{\hat{E}}[\lambda
X]=\lambda\mathbb{\hat{E}}[X]\ \ \ \text{for}\ \lambda\geq0$.
\end{description}

The triple $(\Omega,\mathcal{H},\mathbb{\hat{E}})$ is called a sublinear
expectation space.
\end{definition}

Set $\Omega_{T}:=C_{0}([0,T];\mathbb{R}^{d})$ the space of all $\mathbb{R}%
^{d}$-valued continuous paths $(\omega_{t})_{t\geq0}$ starting from origin,
equipped with the supremum norm. Denote by $\mathcal{B}(\Omega_{T})$ the Borel
$\sigma$-algebra of $\Omega_{T}$ and $B_{t}(\omega):=\omega_{t}$ the canonical
mapping. For each $t\in\lbrack0,T]$, we set $L_{ip}(\Omega_{t}):=\{\varphi
(B_{t_{1}},\ldots,B_{t_{k}}):k\in\mathbb{N},t_{1},\ldots,t_{k}\in
\lbrack0,t],\varphi\in C_{b.Lip}(\mathbb{R}^{k\times d})\}.$

Let $G:\mathbb{S}(d)\rightarrow\mathbb{R}$ be a given monotonic and sublinear
function. Peng constructed the sublinear $G$-expectation space  $(\Omega_T,L_{ip}(\Omega_{T}),\hat{\mathbb{E}%
})$, and under $\hat{\mathbb{E}}$, the
canonical process $B_{t}=(B_{t}^{1},\cdots,B_{t}^{d})$ is called a
$d$-dimensional $G$-Brownian motion.
The conditional $G$-expectation for $X=\varphi(B_{t_{1}}, B_{t_{2}}-B_{t_{1}%
},\cdots,B_{t_{n}}-B_{t_{n-1}})$ at $t=t_{j}$, $1\leq j\leq n$ is defined by
\[
\hat{\mathbb{E}}_{t_{j}}[X]:=\phi(B_{t_{1}}, B_{t_{2}}-B_{t_{1}}%
,\cdots,B_{t_{j}}-B_{t_{j-1}}),
\]
where $\phi(x_{1}, \cdots,x_{j})=\hat{\mathbb{E}}[\varphi(x_{1}, \cdots,x_{j},
B_{t_{j+1}}-B_{t_{j}},\cdots,B_{t_{n}}-B_{t_{n-1}})]$.

For each $p\geq1$, we denote by $L_{G}^{p}(\Omega_{t})$ the completion of
$L_{ip}(\Omega_{t})$ under the norm $||X||_{p}:=(\hat{\mathbb{E}}%
[|X|^{p}])^{1/p}$. The $G$-expectation $\hat{\mathbb{E}}[\cdot]$ and
conditional $G$-expectation $\hat{\mathbb{E}}_{t}[\cdot]$ can be extended
continuously to $L_{G}^{1}(\Omega_{T})$.

The following is the representation theorem for $G$-expectation.

\begin{theorem}
(\cite{D-H-P,H-P})\label{DHP representation} There exists a family
$\mathcal{P}$ of weakly compact probability measures on $(\Omega
_{T},\mathcal{B}(\Omega_{T}))$ such that
\[
\hat{\mathbb{E}}[X]=\sup_{P\in\mathcal{P}}E_{P}[X],\qquad\text{for each}\ X\in
L_{G}^{1}(\Omega_{T}).
\]
$\mathcal{P}$ is called a set that represents $\hat{\mathbb{E}}$.
\end{theorem}

Given $\mathcal{P}$ that represents $\hat{\mathbb{E}}$, we define the
capacity
\[
c(A):=\sup_{P\in\mathcal{P}}P(A),\ \ \ \ \text{for each}\ A\in\mathcal{B}%
(\Omega_{T}).
\]
A set $A\in\mathcal{B}(\Omega_{T})$ is said to be {polar} if $c(A)=0$. A
property is said to {hold \textquotedblleft quasi-surely\textquotedblright}
({q.s.}) if it holds outside a polar set. In the following, we do not
distinguish between two random variables $X$ and $Y$ if $X=Y$ q.s.

Set
\[
\mathcal{L}(\Omega_{T}):=\{X\in\mathcal{B}(\Omega_{T}):E_{P}[X]\ \text{exists
	for each }\ P\in\mathcal{P}\}.
\]
We extend the $G$-expectation to $\mathcal{L}(\Omega_{T})$, still denote it by
$\hat{\mathbb{E}}$, by setting
\[
\hat{\mathbb{E}}[X]:=\sup_{P\in\mathcal{P}}E_{P}[X],\ \ \ \ \text{for}%
\ X\in\mathcal{L}(\Omega_{T}).
\]
Then clearly, $L_{G}^{p}(\Omega_{T})\subset\mathcal{L}(\Omega_{T})$.

\begin{definition}A real function $X$ on $\Omega_T$ is said to be quasi-continuous if for each $\varepsilon>0$,
	there exists an open set $O$ with $c(O)<\varepsilon$ such that
	$X|_{O^{c}}$ is continuous.
\end{definition}
\begin{definition}  We say that  $X:\Omega_T\mapsto\mathbb{R}$ has a quasi-continuous
	version if there exists a quasi-continuous function
	$Y:\Omega_T\mapsto\mathbb{R}$ such that $X = Y$, q.s..
\end{definition}
Then we have the following characterization of the space   $L^p_G(\Omega_T)$, which can be seen as a counterpart of Lusin's theorem in the nonlinear expectation theory.
\begin{theorem}
	\label{LG characterization}For each $p\geq 1$, we have
	\begin{align*}
	L_{G}^{p}(\Omega_T)=\{X\in \mathcal{B}(\Omega_T)\ :\ \
	&\lim\limits_{N\rightarrow\infty}\mathbb{\hat{E}}[|X|^pI_{|X|\geq
		N}]=0 \
	\text{and}\  X\ \text {has a quasi-continuous version}\}.
	\end{align*}
\end{theorem}

\begin{definition}
A process is a family of random variables $X=(X_{t})_{t\in\lbrack0,T]}$ such
that for all $t\in\lbrack0,T]$, $X_{t}\in L_{G}^{1}(\Omega_{T}).$ We say that
process $Y$ is a modification of process $X$ if for each $t\in\lbrack0,T],$
$X_{t}=Y_{t},$ q.s.
\end{definition}

The following is the Kolmogorov criterion for continuous modification with
respect to capacity

\begin{lemma}
\label{modification lemma} Let $(X_{t})_{t\in\lbrack0,T]}$ be a process taking
Banach values. Assume that there exist positive constants $\alpha,\beta$ and
$c$ such that
\[
\hat{\mathbb{E}}[|X_{t}-X_{s}|^{\alpha}]\leq c|t-s|^{1+\beta}.
\]
Then $X$ admits a continuous modification $\tilde{X}$ such that
\[
\hat{\mathbb{E}}\left[  \left(  \sup_{s\neq t}\displaystyle\frac{|\tilde
{X}_{t}-\tilde{X}_{s}|}{|t-s|^{\gamma}}\right)  ^{\alpha}\right]  <\infty,
\]
for every $\gamma\in(0,\beta/\alpha)$.
\end{lemma}

Now we give the definition of $G$-martingales.

\begin{definition}
A process $\{M_{t}\}$ is called a $G$-martingale if $M_{t}\in L_{G}^{1}%
(\Omega_{t})$ and $\hat{\mathbb{E}}_{s}(M_{t})=M_{s}$ for any $s\leq t$. If
$\{M_{t}\}$ and $\{-M_{t}\}$ are both $G$-martingales, we call $\{M_{t}\}$ a
symmetric $G$-martingale.
\end{definition}

\begin{remark}
\upshape{
If $M$ is a symmetric $G$-martingale, then under each $P$, it is a classical martingale.}
\end{remark}

In the following, we give the stochastic calculus with respect to a kind of
martingales as well as its quadratic variation process. In this paper, we
always assume that $M$ is a symmetric martingale satisfying:

\begin{description}
\item[(H)] $M_{t}\in{L}_{G}^{2}(\Omega_{t})$ for each $t\geq0$ and there
exists a nonnegative constant $\Lambda>0$ such that
\begin{equation}
\hat{\mathbb{E}}_{t}[|M_{t+s}-M_{t}|^{2}]\leq\Lambda s,\ \ \ \ \text{ for each
}t,s\geq0. \label{M}%
\end{equation}
\end{description}
\begin{remark}
\upshape{	For $M_{t}\in{L}_{G}^{2}(\Omega_{t})$, by the $G$-martingale representatin theorem (see, e.g., \cite{Song1,STZ}), $M_{t}$ can be represented as the integral of $G$-Brownian motion. From this, we know that $M_{t}$ is continuous.}
\end{remark}
For each $T>0$ and $p\geq1$, we define
\begin{align*}
M_{G}^{p,0}(0,T):=  &  \{\eta=\sum_{j=0}^{N-1}\xi_{j}(\omega)I_{[t_{j}%
,t_{j+1})}(t):N\in\mathbb{N},\ 0\leq t_{0}\leq t_{1}\leq\cdots\leq t_{N}\leq
T,\\
&  \ \xi_{j}\in L_{G}^{p}(\Omega_{t_{j}}),\ j=0,1\cdots,N\}.
\end{align*}
For each $\eta\in M_{G}^{p,0}(0,T)$, set the norm $\Vert\eta\Vert_{M_{G}^{p}%
}:=(\hat{\mathbb{E}}[\int_{0}^{T}|\eta_{t}|^{p}dt])^{\frac{1}{p}}$ and denote
by $M_{G}^{p}(0,T)$ the completion of $M_{G}^{p,0}(0,T)$ under $\Vert
\cdot\Vert_{M_{G}^{p}}$.

For $\eta\in M_{G}^{2,0}(0,T)$, define the stochastic integral with respect to
$M$ by
\[
I(\eta)=\int_{0}^{T}\eta_{t}dM_{t}:=\sum_{j=0}^{N-1}\xi_{t_{j}}(M_{t_{j+1}%
}-M_{t_{j}}): \ M_{G}^{2,0}(0,T)\rightarrow{L}_{G}^{2}(\Omega_{T}).
\]
The proof of following lemma is the same as that of Lemma 3.5 in Chap. III of  \cite{Peng
1}, so we omit it.

\begin{lemma}
\label{control lemma} For each $\eta\in M_{G}^{2,0}(0,T)$, we have
\begin{equation}
\hat{\mathbb{E}}[|\int_{0}^{T}\eta_{t}dM_{t}|^{2}]\leq C\hat{\mathbb{E}}%
[\int_{0}^{T}|\eta_{t}|^{2}dt]. \label{control eq}%
\end{equation}
for some constant $C\geq0$.
\end{lemma}

By the above lemma, we can extend the integral continuously to $M_{G}%
^{2}(0,T)$. For $\eta\in M_{G}^{2}(0,T)$, there exists a sequence $\{\eta
^{n}\}\subset M_{G}^{2,0}(0,T)$ such that $\eta^{n}\rightarrow\eta$ in $M_{G}%
^{2}(0,T)$. By Lemma
\ref{control lemma}, $\int_{0}^{T}\eta_{t}^{n}dM_{t}$ is Cauchy sequence in
${L}_{G}^{2}(\Omega_{T})$, we define
\[
\int_{0}^{T}\eta_{t}dM_{t}:={L}_{G}^{2}-\lim_{n\rightarrow\infty}\int_{0}%
^{T}\eta_{t}^{n}dM_{t}.
\]
It's easy to see $\int_{0}^{t}\eta_{s}dM_{s}$ is a symmetric $G$-martingale
and (\ref{control eq}) still holds.

Next we consider the quadratic variation of $M$. Let $\pi_{t}^{N}=\{t_{0}%
^{N},...,t_{N}^{N}\}$ be a partition of $[0,t]$ and denote
\[
\mu(\pi_{t}^{N}):=\max\{|t_{j+1}^{N}-t_{j}^{N}|:j=0,1,\cdots,N-1\}.
\]
Consider
\begin{align*}
\sum_{j=0}^{N-1}(M_{t_{j+1}^{N}}-M_{t_{j}^{N}})^{2}  &  =\sum_{j=0}%
^{N-1}(M_{t_{j+1}^{N}}^{2}-M_{t_{j}^{N}}^{2})-2\sum_{j=0}^{N-1}M_{t_{j}^{N}%
}(M_{t_{j+1}^{N}}-M_{t_{j}^{N}})\\
&  =M_{t}^{2}-2\sum_{j=0}^{N-1}M_{t_{j}^{N}}(M_{t_{j+1}^{N}}-M_{t_{j}^{N}}).
\end{align*}
Letting $\mu(\pi_{t}^{N})\rightarrow0$, the right side converges to $M_{t}%
^{2}-2\int_{0}^{t}M_{s}dM_{s}$ in ${L}_{G}^{2}(\Omega_{T})$. So
\[
\sum_{j=0}^{N-1}(M_{t_{j+1}^{N}}-M_{t_{j}^{N}})^{2}{\longrightarrow}M_{t}%
^{2}-2\int_{0}^{t}M_{t}dM_{t}\ \ \ \ \text{in}\ {L}_{G}^{2}(\Omega_{T}).
\]
We call this limit the quadratic variation of $M$ and denote it by $\langle
M\rangle_{t}$. By the definition of $\langle M\rangle_{t},$ it is easy to
obtain, for each $t,s\geq0$,
\begin{equation}
\hat{\mathbb{E}}_{t}[|\langle M\rangle_{t+s}-\langle M\rangle_{t}|
]=\hat{\mathbb{E}}_{t}[\langle M\rangle_{t+s}-\langle M\rangle_{t}%
]=\hat{\mathbb{E}}_{t}[M_{t+s}^{2}-M_{t}{}^{2}]=\hat{\mathbb{E}}_{t}%
[|M_{t+s}-M_{t}|^{2}]\leq\Lambda s. \label{M control}%
\end{equation}

\begin{remark}
\upshape{ Note that $\langle
		M\rangle_{t}$ is q.s. defined, and under each $P\in\mathcal{P}$, it is also the classical quadratic variation $\langle
		M\rangle^P_{t}$ of martingale $M$.
}
\end{remark}

\section{Main results}

We first introduce a bigger integrand space for the stochastic calculus of $G$-martingales $M$. This space plays an important role in the construction of local time.

For $p\geq1$ and $\eta\in M_{G}^{p,0}(0,T)$, we define a new norm $\left\vert
\left\vert \eta\right\vert \right\vert _{\bar{M}_{G}^{p}}=(
\hat{\mathbb{E}}[  \int_{0}^{T}\left\vert \eta_{t}\right\vert
^{p}d\langle M\rangle_{t}]  )  ^{\frac{1}{p}}$ and denote the
completion of $M_{G}^{p,0}(0,T)$ under the norm $\left\vert \left\vert
\cdot\right\vert \right\vert _{\bar{M}_{G}^{p}}$ by $\bar{M}_{G}^{p}(0,T)$.

We have the following result concerns the relationship between spaces $M_{G}^{p}(0,T)$ and $\bar{M}_{G}
^{p}(0,T)$.
\begin{lemma}
We have, 
\begin{equation}
\left\vert \left\vert \eta\right\vert \right\vert _{\bar{M}_{G}^{p}}%
\leq\Lambda^{\frac{1}{p}}\left\vert \left\vert \eta\right\vert \right\vert
_{M_{G}^{p}}\ \text{for each}\ \eta\in M_{G}^{p,0}(0,T), \ \ \ \text{and}\ \ \ \ M_{G}^{p}(0,T)\subset\bar{M}_{G}%
^{p}(0,T). \label{234423}%
\end{equation}
If moreover there exists a constant $0<\lambda\leq\Lambda$ such that
$\hat{\mathbb{E}}_{t}[|M_{t+s}-M_{t}|^{2}]\geq\lambda s$, then the quadratic
variation of martingales is non-degenerate, i.e., $\hat{\mathbb{E}}%
_{t}[\langle M\rangle_{t+s}-\langle M\rangle_{t}]\geq\lambda s$, which
implies
\[
\left\vert \left\vert \eta\right\vert \right\vert _{\bar{M}_{G}^{p}}%
\geq\lambda^{\frac{1}{p}}\left\vert \left\vert \eta\right\vert \right\vert
_{M_{G}^{p}}\ \text{for each}\ \eta\in M_{G}^{p,0}(0,T), \  \ \text{and}\ \ \ \ M_{G}^{p}(0,T)=\bar{M}_{G}^{p}(0,T).
\]

\end{lemma}

\begin{proof}
We just prove (\ref{234423}) since the proof for the second part is just
similar. We only need to prove $\left\vert \left\vert \eta\right\vert
\right\vert _{\bar{M}_{G}^{p}}\leq\Lambda^{\frac{1}{p}}\left\vert \left\vert
\eta\right\vert \right\vert _{M_{G}^{p}}\ $\ for $\eta=\sum_{j=0}^{N-1}\xi
_{j}(\omega)I_{[t_{j},t_{j+1})}(t)\in M_{G}^{p,0}(0,T).$ By the sub-linearity
of $\hat{\mathbb{E}}$, we have
\begin{equation}%
\begin{split}
\hat{\mathbb{E}}[  \int_{0}^{T}\left\vert \eta_{t}\right\vert
^{p}d\langle M\rangle_{t}]   &  =\hat{\mathbb{E}}[  \sum
_{j=0}^{N-1}|\xi_{j}(\omega)|^{p}(\langle M\rangle_{t_{j+1}}-\langle
M\rangle_{t_{j}})] \\
&  \leq\hat{\mathbb{E}}[  \sum_{j=0}^{N-1}|\xi_{j}(\omega)|^{p}(\langle
M\rangle_{t_{j+1}}-\langle M\rangle_{t_{j}}-\Lambda(t_{j+1}-t_{j}))]
+\Lambda\hat{\mathbb{E}}[  \sum_{j=0}^{N-1}|\xi_{j}(\omega)|^{p}%
(t_{j+1}-t_{j})]
\end{split}
\label{098765}%
\end{equation}
Note that
\begin{align*}
\hat{\mathbb{E}}[  \sum_{j=0}^{N-1}|\xi_{j}(\omega)|^{p}(\langle
M\rangle_{t_{j+1}}-\langle M\rangle_{t_{j}}-\Lambda(t_{j+1}-t_{j}))]
&  \leq\sum_{j=0}^{N-1}\hat{\mathbb{E}}[  |\xi_{j}(\omega)|^{p}(\langle
M\rangle_{t_{j+1}}-\langle M\rangle_{t_{j}}-\Lambda(t_{j+1}-t_{j}))] \\
&  =\sum_{j=0}^{N-1}\hat{\mathbb{E}}[  |\xi_{j}(\omega)|^{p}%
\hat{\mathbb{E}}_{t_{j}}[\langle M\rangle_{t_{j+1}}-\langle M\rangle_{t_{j}%
}-\Lambda(t_{j+1}-t_{j})]] \\
&  \leq0.
\end{align*}
Combining this with (\ref{098765}), we get the disired result.
\end{proof}

For $\eta\in\bar{M}_{G}^{2,0}(0,T)$, by a similar analysis as in Proposition
4.5 in Chap. III\ of \cite{Peng 1}, we have $$\hat{\mathbb{E}}[ ( \int%
_{0}^{T}\eta_{t}dM_{t})^2] \leq\hat{\mathbb{E}}[  \int_{0}%
^{T}\left\vert \eta_{t}\right\vert ^{2}d\langle M\rangle_{t}].$$ Then
the definition of integral $\int_{0}^{T}\eta_{t}dM_{t}$ can be extended continuously to $\bar{M}_{G}%
^{2}(0,T)$. Moreover, on $M_{G}^{2}(0,T)$, this definition 
 coincides with the one in Section 1.

In the following, $C$ always denotes a generic constant which is free to vary from line to line. By a standard argument, we can obtain a regular version of the stochastic integral.
\begin{proposition}
For $\eta\in\bar{M}_{G}^{2}(0,T),$ there exists a modification of
$\int_{0}^{t}\eta_{s}dM_{s}$ such that $t\rightarrow\int_{0}^{t}\eta_{s}%
dM_{s}$ is continuous.
\end{proposition}

\begin{proof}
Denote $I(\eta)_{t}=\int_{0}^{t}\eta_{s}dM_{s}$. We can take a sequence
$\eta^{n}\in\bar{M}_{G}^{2,0}(0,T)$ such that $\eta^{n}\rightarrow\eta$ in
$\bar{M}_{G}^{2}(0,T).$ It is easy to see that $t\rightarrow\int_{0}^{t}%
\eta_{s}^{n}dM_{s}$ is continuous. By B-D-G inequality,
\begin{equation}\label{08878}
\mathbb{\hat{E}}[\sup_{0\leq t\leq T}|I(\eta^{n})_{t}-I(\eta^{m})_{t}%
|^{2}]\leq C\mathbb{\hat{E}}[\int_{0}^{T}|\eta^{n}-\eta^{m}|^{2}d\langle
M\rangle_{t}]
\end{equation}
From Markov inequality (see \cite{D-H-P}), for each $a>0,$
\begin{equation}\label{67675}
\mathbb{\hat{E}}[\sup_{0\leq t\leq T}|I(\eta^{n})_{t}-I(\eta^{m})_{t}|\geq
a]\leq\frac{1}{a^{2}}\mathbb{\hat{E}}[\sup_{0\leq t\leq T}|I(\eta^{n}%
)_{t}-I(\eta^{m})_{t}|^{2}]
\end{equation}
Combining (\ref{08878}) and (\ref{67675}), and using the Borel-Cantelli lemma (see \cite{Peng 1}), one can extract a subsequence
$I(\eta^{n_{k}})_{t}$ converging q.s. uniformly. We denote this limit by
$Y_{t}$, and it is a continuous modification of $I(\eta)_{t}.$
\end{proof}

\begin{remark}
\upshape{Henceforth, we will consider only the continuous modifications of $\int_{0}%
^{t}\eta_{s}dM_{s}.$}
\end{remark}

Let us state some characterization results for the space $\bar{M}_{G}%
^{p}(0,T).$ which are important for our future discussion.

\begin{lemma}
\label{lemma-1} Assume $X\in\bar{M}_{G}^{p}(0,T)$. Then for each $\varphi\in
C_{b.Lip}(\mathbb{R})$, we have $\varphi(X_{t})_{0\leq t\leq T}\in\bar{M}%
_{G}^{p}(0,T).$
\end{lemma}

\begin{proof}
We can find a sequence $X_{t}^{n}=\sum_{j=0}^{N_{n}-1}\xi_{j}^{n}%
(\omega)I_{[t_{j}^{n},t_{j+1}^{n})}(t)$, where $\xi_{j}\in L_{G}^{p}%
(\Omega_{t_{j}^{n}})$, such that $X^{n}\rightarrow X$ under the norm
$\left\vert \left\vert \cdot\right\vert \right\vert _{\bar{M}_{G}^{p}}$. Note
that $\varphi(\xi_{j}^{n}(\omega))\in L_{G}^{p}(\Omega_{t_{j}^{n}})$ by
Theorem \ref{LG characterization}, then we have
\[
\varphi(X_{t}^{n})=\sum_{j=0}^{N_{n}-1}\varphi(\xi_{j}^{n}(\omega
))I_{[t_{j}^{n},t_{j+1}^{n})}(t)\in\bar{M}_{G}^{p,0}(0,T).
\]
Then the desired result follows from the observation that%
\[
\hat{\mathbb{E}}[\int_{0}^{T}|\varphi(X_{t})-\varphi(X_{t}^{n})|^{p}d\langle
M\rangle_{t}]\leq L_{\varphi}^{p}\hat{\mathbb{E}}[\int_{0}^{T}|X_{t}-X_{t}%
^{n}|^{p}d\langle M\rangle_{t}]\rightarrow0,\ \ \ \ \text{as}\ n\rightarrow
\infty,
\]
where $L_{\varphi}$ is the Lipschitz constant of $\varphi.$
\end{proof}

\begin{proposition}
	\label{lemma-2} Assume $\eta\in\bar{M}_{G}^{p}(0,T)$. Then $$\hat{\mathbb{E}}%
	[\int_{0}^{T}|\eta_t|^pI_{\{|\eta_t|> N\}}d\langle M\rangle_{t}]\rightarrow0,\  \ \ \ \text{as}\ N\rightarrow0.$$
\end{proposition}
\begin{proof}
	It suffices to prove the case that $\eta_{t}=\sum_{j=0}^{n-1}\xi_{j}
	(\omega)I_{[t_{j},t_{j+1})}(t)\in M^{p,0}_G(0,T)$, where $\xi_j\in L^p_G(\Omega_{t_j})$. We take  bounded, continuous functions $\varphi_N$ such that $I_{\{|x|>N \}}\leq \varphi_N\leq I_{\{|x|>{N-1} \}}$.
	Then by Theorem \ref{LG characterization},
	\begin{align*}
	\hat{\mathbb{E}}%
	[\int_{0}^{T}|\eta_t|^pI_{\{|\eta_t|> N\}}d\langle M\rangle_{t}]&=\hat{\mathbb{E}}[\sum_{i=1}^{n-1}|\xi_j|^pI_{\{|\xi_j|>N \} }(\langle M\rangle_{t_{j+1}}-\langle M\rangle_{t_{j}%
	})]\\&\leq \sum_{i=1}^{n-1}\hat{\mathbb{E}}[|\xi_j|^p\varphi_N(\xi_j)(\langle M\rangle_{t_{j+1}}-\langle M\rangle_{t_{j}%
})]\\
&=\sum_{i=1}^{n-1}\hat{\mathbb{E}}[|\xi_j|^p\varphi_N(\xi_j)\hat{\mathbb{E}}_{t_j}[\langle M\rangle_{t_{j+1}}-\langle M\rangle_{t_{j}%
}]]\\
&\leq \Lambda\sum_{i=1}^{n-1}\hat{\mathbb{E}}[|\xi_j|^pI_{\{|\xi_j|>N-1 \} }]({t_{j+1}}-{t_{j}%
})\rightarrow 0,\ \text{as}\ N\rightarrow \infty.
	\end{align*}
	\end{proof}

Recall that we always assume that $(M_{t}%
)_{t\geq0}$ is a symmetric $G$-martingale satisfying (H). The following Krylov's estimate can be used to show  that  a kind of processes belong to  $\bar{M}%
_{G}^{p}(0,T)$.
\begin{theorem}
\label{Krylov}(Krylov's estimate)\cite{Krylov-2, Me, Situ} There exists
some constant $C$ depending on $\Lambda$ and $T$ such that, for each $p\geq1$
and Borel function $g$,
\[
\mathbb{\hat{E}}[  \int_{0}^{T}|g(M_{t})|d\langle M\rangle_{t}]
\leq C(  \int_{\mathbb{R}}|g(x)|^{p}dx)  ^{1/p}.
\]

\end{theorem}

\begin{proof}
We outline the proof for the convenience of readers. For any $P\in\mathcal{P}%
$, $(M_{t})_{t\geq0}$ is a martingale. By H\"{o}lder's inequality,
\begin{equation}
\label{234243}E_{P}[  \int_{0}^{T}|g(M_{t})|d\langle M\rangle_{t}]
\leq C_{1}(  E_{P}[  \int_{0}^{T}|g(M_{t})|^{p}d\langle M\rangle
_{t}]  )  ^{1/p},
\end{equation}
where $C_{1}=(\mathbb{\hat{E}}[\langle M\rangle_{T}])^{(p-1)/p}$. Let
$L_{t}^{P}(a)$ be the correponding local time at $a$ of $M$ under $P$. By
 the classical Tanaka formula (see, e.g., \cite{RW}),
\[
L_{T}^{P}(a)=|M_{T}-a|-|M_{0}-a|-\int_{0}^{T}\text{sgn}(M_{t}-a)dM_{t},\text{
}P\text{-a.s.} \label{new-nwww-2}%
\]
Taking expectation on both sides, we get%
\[
0\leq E_{P}[L_{T}^{P}(a)]=E_{P}[|M_{T}-a|]-|M_{0}-a|\leq E_{P}[|M_{T}%
-M_{0}|]\leq\mathbb{\hat{E}}[|M_{T}-M_{0}|].
\]
Applying the occupation time formula under $P$, we obtain
\begin{equation}
\label{97243}E_{P}[\int_{0}^{T}|g(M_{t})|^{p}d\langle M\rangle_{t}]=E_{P}%
[\int_{\mathbb{R}}|g(a)|^{p}L_{T}^{P}(a)da]=\int_{\mathbb{R}}|g(a)|^{p}%
E_{P}[L_{T}^{P}(a)]da\leq C_2\int_{\mathbb{R}}|g(a)|^{p}da,
\end{equation}
where $C_{2}=\mathbb{\hat{E}}[|M_{T}-M_{0}|]$. Combining (\ref{234243}) and
(\ref{97243}), we have
\[
E_{P}[  \int_{0}^{T}|g(M_{t})|d\langle M\rangle_{t}]  \leq
C(\int_{\mathbb{R}}|g(a)|^{p}da)^{\frac1p}.
\]
Note that $C$ is independent of $P$, the desired result follows by taking
supremum over $P\in\mathcal{P}$ in the above inequality.
\end{proof}

\begin{lemma}
\label{a.e. identity} Assume $\varphi^{\prime},\varphi$ are Borel measurable
and$\ \varphi^{\prime}=\varphi\ a.e.$ Then for each $p\geq1$, we have
$\varphi^{\prime}(M_{\cdot})=\varphi(M_{\cdot})$ in $\bar{M}_{G}^{p}(0,T)$,
i.e., $\left\vert \left\vert \varphi^{\prime}(M_{\cdot})-\varphi(M_{\cdot
})\right\vert \right\vert _{\bar{M}_{G}^{p}}=0$.
\end{lemma}

\begin{proof}
By Theorem \ref{Krylov}, we get $\mathbb{\hat{E}}[\int_{0}^{T}|\varphi
^{\prime}-\varphi|^{p}(M_{t})d\langle M\rangle_{t}]\leq C\Vert\varphi^{\prime
}-\varphi\Vert^p_{L^{p}(\mathbb{R})}=0.$
\end{proof}

The following is a kind of dominated convergence result for the $G$-martingales.
\begin{proposition}
\label{myq13} Assume $(\varphi^{n})_{n\geq1}$ is a sequence of Borel
measurable functions such that $\varphi^{n}$ is linear growth uniformly, i.e.,
$|\varphi^{n}(x)|\leq C(1+|x|)$, $n\geq1$ for some constants $C$. If
$\varphi^{n}\rightarrow\varphi$ a.e., then
\[
\lim\limits_{n\rightarrow\infty}\mathbb{\hat{E}}[\int_{0}^{T}|\varphi
^{n}(M_{t})-\varphi(M_{t})|^{2}d\langle M\rangle_{t}]=0.
\]

\end{proposition}

\begin{proof}
By Lemma \ref{a.e. identity}, without loss of generality, we may assume
$|\varphi(x)|\leq C(1+|x|)$. For any $N>0$, we have
\begin{equation}%
\begin{split}
\hat{\mathbb{E}}[\int_{0}^{T}|\varphi^{n}(M_{t})-\varphi(M_{t})|^{2}d\langle
M\rangle_{t}]\leq &  \hat{\mathbb{E}}[\int_{0}^{T}|\varphi^{n}(M_{t}%
)-\varphi(M_{t})|^{2}I_{\{|M_{t}|\leq N\}}d\langle M\rangle_{t}]\\
&  +\hat{\mathbb{E}}[\int_{0}^{T}|\varphi^{n}(M_{t})-\varphi(M_{t}%
)|^{2}I_{\{|M_{t}|>N\}}d\langle M\rangle_{t}].
\end{split}
\label{987758}%
\end{equation}
According to Theorem \ref{Krylov}, we can find a constant $C^{\prime}$ such
that
\[
\hat{\mathbb{E}}[\int_{0}^{T}|\varphi^{n}(M_{t})-\varphi(M_{t})|^{p}%
I_{\{|M_{t}|\leq N\}}d\langle M\rangle_{t}]\leq C^{\prime}\int_{\{|x|\leq
N\}}|\varphi^{n}(x)-\varphi(x)|^{p}dx,
\]
which converges to $0$, as $n\rightarrow\infty$ by the Lesbesgue's dominated
convergence theorem. On the other hand, 
note that $M_{t}=\mathbb{\hat{E}}_{t}[M_{T}]$, by an approximation argument, we have $(M_{t})_{t\leq T}\in{M}_{G}^{2}(0,T)\subset\bar{M}_{G}^{2}(0,T).$ 
Then, by the linear growth condition on $\varphi^{n}$ and $\varphi,$ and  Proposition \ref{lemma-2},
\begin{align*}
\hat{\mathbb{E}}[\int_{0}^{T}|\varphi^{n}(M_{t})-\varphi(M_{t})|^{2}%
I_{\{|M_{t}|> N\}}d\langle M\rangle_{t}]  &  \leq C\hat{\mathbb{E}}%
[\int_{0}^{T}(1+|M_{t}|)^{2}I_{\{|M_{t}|> N\}}d\langle M\rangle_{t}]\rightarrow0,\ \text{as}\ N\rightarrow0.
\end{align*}
First letting $n\rightarrow\infty$ and then letting $N\rightarrow\infty$ in
(\ref{987758}), we get the desired result.
\end{proof}

By Krylov's estimates and Proposition \ref{myq13}, we can show that $\bar{M}_{G}^{2}(0,T)$ contains a lot of processes that we may interest. Such kind of processes are important for the construction of local time.

\begin{proposition}
\label{Mtilde} For
each Borel measurable function $\varphi$ of linear growth, we have
$(\varphi(M_{t}))_{t\leq T}\in\bar{M}_{G}^{2}(0,T).$
\end{proposition}

\begin{proof}
We take a
sequence of bounded, Lipschitz continuous functions $(\varphi^{n})_{n\geq1}$,
such that $\varphi^{n}$ converges to $\varphi$ a.e. and $|\varphi^{n}(x)|\leq
C(1+|x|)$. Then by Theorem \ref{myq13}, we have
\[
\lim\limits_{n\rightarrow\infty}\mathbb{\hat{E}}[\int_{0}^{T}|\varphi
^{n}-\varphi|^{2}(M_{t})d\langle M\rangle_{t}]=0.
\]
Since $(\varphi^{n}(M_{t}))_{t\leq T}\in\bar{M}_{G}^{2}(0,T)$ for each $n$ by
Lemma \ref{lemma-1}, we derive that $(\varphi(M_{t}))_{t\leq T}\in\bar{M}%
_{G}^{2}(0,T)$, and this completes the proof.
\end{proof}

Now we can define the local time of $G$-martingale $M$. For each $P\in\mathcal{P}%
$, by the classical Tanaka formula under $P$,
\begin{equation}
|M_{t}-a|=|M_{0}-a|+\int_{0}^{t}\text{sgn}(M_{s}-a)dM_{s}+L_{t}^{P}(a),\text{
}P\text{-a.s.}, \label{new-nwww-4}%
\end{equation}
where $L_{t}^{P}(a)$ is the local time of martingale $M_{t}$ at $a$ under $P$. According to Proposition
\ref{Mtilde}, we have $($sgn$(M_{s}-a))_{s\leq t}\in\bar{M}_{G}^{2}(0,t)$.
This implies that $\int_{0}^{t}$sgn$(M_{s}-a)dM_{s}\in L_{G}^{2}(\Omega_{t})$. We define the local time for $G$-martingale $M$ by
\[
L_{t}(a):=|M_{t}-a|-|M_{0}-a|-\int_{0}^{t}\text{sgn}(M_{s}-a)dM_{s}\in
L_{G}^{2}(\Omega_{t}).
\]
Then (\ref{new-nwww-4}) gives that
\begin{equation}
L_{t}(a)=L_{t}^{P}(a),\ P\text{-a.s.} \label{897834}%
\end{equation}

The local time always possesses a bicontinuous modification.
\begin{theorem}
There exists a modification of the process $\{L_{t}(a):t\in\lbrack
0,T],a\in\mathbb{R}\}$ such that $L_{t}(a)$ is bicontinuous, i.e., the map
$(a,t)\rightarrow L_{t}(a)$ is continuous.
\end{theorem}

\begin{proof}
It suffices to prove that $\hat{M}_{t}^{a}:=\int_{0}^{t}\text{sgn}%
(M_{s}-a)dM_{s}$ has such kind of modification. Let any $N>0$ be given. For each integer
$n\geq1$, we define stopping time
\[
\tau_{N}=\inf\{s:|M_{s}-M_{0}|^{n}+\langle M\rangle_{s}^{\frac{n}{2}}\geq N\}.
\]
Denote $\bar{M}_{t}=M_{\tau_{N}\wedge t}$. Under each $P\in\mathcal{P}$, $\bar{M}_{t}$
is a martingale. We  denote the correponding local time of $\bar{M}$ by
$\overline{L}_{t}^{P}(a).$ Then by classical B-D-G inequality,
\[
E_{P}[|\overline{L}_{T}^{P}(a)|^{n}]\leq C_{n}(|\bar{M}_{T}-\bar{M}_{0}%
|^{n}+|\langle\bar{M}\rangle_{T}^{P}|^{\frac{n}{2}})\leq C_{n}N,
\]
where $C_n$ is a constant depending on $n$ and may vary from line to line.
For $x<y,$ from occupation formula under $P$ and H\"{o}lder's inequality, we
have
\begin{align*}
&  E_{P}[\sup_{0\leq t\leq T}|\int_{0}^{t}\text{sgn}(\bar{M}_{u}-x)d\bar
{M}_{u}-\int_{0}^{t}\text{sgn}(\bar{M}_{u}-y)d\bar{M}_{u}|^{2n}]\\
&  \leq C_{n}E_{P}[|\int_{0}^{T}I_{[x,y)}(\bar{M}_{t})d\langle\bar{M}%
\rangle_{t}^{P}|^{n}]\\
&  \leq C_{n}E_{P}[|\int_{x}^{y}\overline{L}_{t}^{P}(a)da|^{n}]\\
&  \leq C_{n}(y-x)^{n}E_{P}[\frac{1}{y-x}\int_{x}^{y}|\overline{L}_{T}%
^{P}(a)|^{n}da]\\
&  \leq C_{n}N(y-x)^{n}.
\end{align*}
Note that, $\hat{M}_{t\wedge\tau_N}^{a}=\int_{0}^{t}$sgn$(\bar{M}_{u}-a)d\bar
{M}_{u},$ $P$-a.s. Thus,%
\[
\hat{\mathbb{E}}[\sup_{0\leq t\leq T}|\hat{M}_{t\wedge\tau_N}^{x}-\hat{M}_{t\wedge\tau_N}^{y}
|^{2n}]\leq C_{n}N(y-x)^{n}.
\]
Applying Lemma \ref{modification lemma} to $$ a\rightarrow \hat{M}_{\cdot\wedge \tau_N}^{a}\in E:=C([0,T];\mathbb{R}),$$ we obtain that   
 $\hat{M}_{t\wedge\tau_N}^{a}$ has a bicontinuous version for each $N$, which implies that
$\hat{M}_{t}^{a}$ has a bicontinuous version.
\end{proof}

Now we give the Tanaka formula for convex functions of $G$-martingales.
\begin{theorem}
Let $f$ be a convex function such that left derivative $f_{-}^{\prime}$
satisfies the linear growth condition. Then%
\begin{equation}
\label{56783}f(M_{t})-f(M_{0})=\int_{0}^{t}f_{-}^{\prime}(M_{s})dM_{s}+\frac{1}{2}%
\int_{\mathbb{R}}L_{t}(a)df_{-}^{\prime}(a),\ \ \ \ \text{q.s.}%
\end{equation}
where $df_{-}^{\prime}$ is the Lebesgue-Stieltjes measure of $f_{-}^{\prime}$. Moreover, the integral $\int_{\mathbb{R}}L_{t}(a)df_{-}^{\prime}(a) \in L_{G}^{1}(\Omega_{t}).$
\end{theorem}

\begin{proof}
According to Proposition \ref{Mtilde}, we have $(f_{-}^{\prime}(M_{s}))_{s\leq T}%
\in\bar{M}_{G}^{2}(0,T)$. Note that, under each $P\in\mathcal{P}$,  $\int_{0}^{t}f_{-}^{\prime}(M_{s})dM_{s}$ is also the stochastic
integral with respect to martingale $M_{t}$ and $L_{t}(a)$ is the local time
of martingale $M_{t}$. By the classical Tanaka formula for martingales, we
have
\[
f(M_{t})-f(M_{0})=\int_{0}^{t}f_{-}^{\prime}(M_{s})dM_{s}+\frac{1}{2}\int_{\mathbb{R}%
}L_{t}(a)df_{-}^{\prime}(a),\ P\text{-a.s.}%
\]
Since the four terms in the above identity both q.s. defined, we deduce that
the above formula holds q.s.

Since convex function $f$ is continuous, we have $f(M_t)$ is quasi-continuous. Moreover,   the linear growth condition of $f_{-}^{\prime}$
implies that
$|f(x)|\leq C(1+|x|+|x|^2)$ by  Problem 3.6.21 (6.46) in  \cite{KS}. Thus,
\begin{align*}
\hat{\mathbb{E}}[|f(M_t)|I_{\{|f(M_t)|>N \} }]\leq C\hat{\mathbb{E}}[(1+|M_t|+|M_t|^2)I_{\{|M_t|>\frac{N}{C} \} }]\rightarrow 0,\ \ \text{as}\ N\rightarrow \infty.
\end{align*}
Then from Theorem \ref{LG characterization}, we deduce that $f(M_t)\in L^1_G(\Omega_t)$, which, together with (\ref{56783}), implies $\int_{\mathbb{R}}L_{t}(a)df_{-}^{\prime}(a) \in L_{G}^{1}(\Omega_{t}).$
\end{proof}

Finally, we list some useful properties of local time, which follow directly from
applying the classical ones under each $P\in\mathcal{P}$.

\begin{proposition}
We have

\begin{description}
\item[\textrm{(i)}] The measure $dL_{t}(a)$ grows only when
$M=a$: $\int_{\mathbb{R}_+}I_{\{M_{t}\neq a\}}dL_{t}(a)=0,$ q.s.;

\item[\textrm{(ii)}] Occupation time formua: for each bounded or positive Borel measurable
function $g$, $\int_{0}^{T}g(M_{t})d\langle M\rangle_{t}=\int_{\mathbb{R}}
g(a)L_{T}^{P}(a)da,$ q.s.;

\item[\textrm{(iii)}] For the bicontinuous version of $L_{t}(a)$, the following representation hold: $$L_{t}(a)=\lim_{\varepsilon
	\downarrow0}\frac{1}{\varepsilon}\int_{0}^{t}I_{[a,a+\varepsilon
	)}(M_{s})d\langle M\rangle_{t}=\lim_{\varepsilon
\downarrow0}\frac{1}{2\varepsilon}\int_{0}^{t}I_{(a-\varepsilon,a+\varepsilon
)}(M_{s})d\langle M\rangle_{t}, q.s.$$
\end{description}
\end{proposition}

\end{document}